\documentclass[12pt,letterpaper,notitlepage]{amsart}
\usepackage{fullpage}
\usepackage{parskip}

\usepackage{amssymb}
\usepackage{amsmath}
\usepackage{amsthm}
\usepackage{amsfonts}

\usepackage[linktocpage=true]{hyperref} 

\usepackage{graphicx}
\usepackage{float}

\usepackage[active]{srcltx} 

\usepackage[applemac]{inputenc}

\usepackage[OT2,OT1]{fontenc}
\newcommand\cyr{%
\renewcommand\rmdefault{wncyr}%
\renewcommand\sfdefault{wncyss}%
\renewcommand\encodingdefault{OT2}%
\normalfont
\selectfont}

\DeclareTextFontCommand{\textcyr}{\cyr}

\newfont{\cmbsy}{cmbsy10}
\newfont{\cmmib}{cmmib10}

\newcommand{\Orden}{\mathop{\hbox{\cmbsy O}}\nolimits}

\def\R{\mathbf{R}}
\def\C{\mathbf{C}}

\def\Re{\mathrm{Re\,}}

\DeclareMathOperator*{\card}{card}

\def\Re{\operatorname{Re}}

\def\medio{{\textstyle \frac{1}{2}}\,}

\def\tt{\alpha}

\newtheorem{theorem}{Theorem}[section]
\newtheorem*{cor}{Corollary}
\newtheorem{corollary}[theorem]{Corollary}

\newtheorem{lemma}[theorem]{Lemma}

\newtheorem*{maintheorem}{Theorem \ref{maintheorem}}

\theoremstyle{definition}
\newtheorem{definition}[theorem]{Definition}

\theoremstyle{remark}

\begin{document}

\title[Distribution of normalized zeros]
{On the distribution  (mod 1) of the normalized zeros of the Riemann Zeta-function}
\author[Arias de Reyna]{J. Arias de Reyna}
\address{Facultad de Matemáticas \\
Univ.~de Sevilla \\
Apdo.~1160
 \\
41080-Sevilla \\
Spain} 
\thanks{Author supported by  MINECO grant MTM2012-30748.}
\email{arias@us.es}

\date{\today}

\begin{abstract}
We consider the problem whether the ordinates of the non-trivial zeros of $\zeta(s)$ 
are uniformly distributed modulo the Gram points, or equivalently, if the normalized
zeros $(x_n)$ are uniformly distributed modulo 1. Odlyzko conjectured this 
to be true. This is far from being proved, even assuming the Riemann hypothesis 
(RH, for short).

Applying the Piatetski-Shapiro $11/12$ Theorem we are able to show that, 
for $0<\kappa<6/5$,  the mean 
value  $\frac1N\sum_{n\le N}\exp(2\pi i \kappa x_n)$ tends to zero. The case $\kappa=1$
is especially interesting. In this case the Prime Number Theorem is sufficient to prove
that the mean value is $0$, but the rate of convergence is slower than for 
other values of $\kappa$.
Also the case $\kappa=1$ seems to contradict the behavior of the first 
two million zeros of $\zeta(s)$. 

We make an effort not to use the RH. So our Theorems
are absolute. We also put forward the interesting question: will the 
uniform distribution of 
the normalized zeros be compatible with the GUE hypothesis?

Let $\rho=\frac12+i\tt$ run through the complex zeros of zeta. We do not assume
the RH so that  $\tt$ may be complex.
For $0<\kappa<\frac65$ we prove that 
\[\lim_{T\to\infty}\frac{1}{N(T)}\sum_{0<\Re\tt\le T}e^{2i\kappa\vartheta(\tt)}=0\]
where $\vartheta(t)$ is the phase of $\zeta(\frac12+it)=e^{-i\vartheta(t)}Z(t)$. 
\end{abstract}

\maketitle


\section{Introduction.}

This paper deals with the distribution of the ordinates of the non-trivial 
zeros of $\zeta(s)$.
This distribution  has 
received much attention. On the assumption of the RH 
Rademacher \cite{Rad} (also \cite[{p.~434--459}]{MR0505096})  showed that the sequence $(\gamma_n)$ is uniformly distributed
modulo 1.  Here $\rho_n=\beta_n+i\gamma_n$ are the zeros of $\zeta(s)$ in the upper half plane,
counted with multiplicity,  ordered by increasing
$\gamma_n$, and ties being broken by ordering $\beta_n$ from smallest to largest. 

Given a sequence of non-negative real numbers $(\gamma_n)$ and another 
strictly increasing sequence of non-negative real numbers $(g_n)$ with $\lim_n g_n=\infty$,  it is said
that the $(\gamma_n)$ are \emph{uniformly distributed modulo $(g_n)$} if the 
numbers $y_n$ defined by 
\[ y_n=\frac{\gamma_n-g_m}{g_{m+1}-g_m},\quad\text{where}\quad \gamma_n\in[g_m,g_{m+1})\]
are uniformly distributed in $[0,1]$, (compare \cite[p.~4]{MR0419394}).

Our problem is whether  the ordinates $(\gamma_n)$ of the zeros of 
$\zeta(s)$ are uniformly distributed modulo the Gram points $(g_n)$. 
In a non published report,  Odlyzko \cite[p.~60]{O} conjectured that the 
ordinates of the zeros $\gamma_n$ 
are not  related  to the Gram points for $\gamma_n$ large.
Hence he conjectured that the normalized zeros will be uniformly distributed 
modulo $1$. This is far from being proved, even assuming the RH. 

Since the Gram points are defined by the equation $\pi n=\vartheta(g_n)$, our problem
is equivalent to the question as to whether the \emph{normalized zeros} $x_n:=\frac{1}{\pi}
\vartheta(\gamma_n)\approx \frac{\gamma_n}{2\pi}
\log\frac{\gamma_n}{2\pi e}$ are uniformly distributed  $\bmod\ 1$. By the asymptotic properties
of $\vartheta(t)$ this is equivalent to the uniform distribution $\bmod\ 1$  of the numbers 
$\frac{\gamma_n}{2\pi}\log\frac{\gamma_n}{2\pi e}$ (see Theorem 1.2 in \cite[{p.3}]{MR0419394}).  The question is especially interesting 
because the normalized zeros have average spacing $1$.  

Hardy and Littlewood \cite[{p. 162--177}]{HL} proved that for any $a>0$, uniformly for $x\in J$, 
where $J$ is
any compact interval of positive numbers, we have
\[\sum_{0<\gamma<T}e^{a\rho\log(-i\rho)}x^\rho\rho^{-\frac{a}{2}}=\Orden(T^{\frac12+\frac{a}{2}})\]
from which, assuming the RH, they derive that for any $a$, $\theta>0$, we have 
\[\sum_{0<\gamma<T}e^{ai\gamma\log (\gamma\theta)}=\Orden(T^{\frac12+\frac{a}{2}}).\]

Recall that $g_n\sim 2\pi n/\log n$. Fujii \cite{MR511699} proved  that for 
any $a>0$ and $b>0$, the sequence $(\gamma_n)$ is
uniformly distributed modulo $((\log n)^a \cdot bn/\log n)$. 

In \cite{MR1179115} Fujii proves, under the RH, that for any positive $\kappa$ and $a$
we have
\[\sum_{0<\gamma\le T}e^{i\kappa \gamma\log\frac{\kappa\gamma}{2\pi e a}}=
-e^{\frac{\pi i}{4}} \frac{\sqrt{a}}{\kappa}
\sum_{n\le (\frac{\kappa T}{2\pi a})^\kappa}
\frac{\Lambda(n)}{n^{\frac12-\frac{1}{2\kappa}}}e^{-2\pi i a n^{\frac1\kappa}}+
\Orden(T^{\frac{\kappa}{2}}\log T \log\log T)+\Orden(T^{\frac12-\frac{\kappa}{2}}\log T).\]

Our main result is an absolute version of this equation. Since we do not assume the RH, we define 
the numbers $\tt$ in such a way that the zeros of zeta are given by
$\rho=\frac12+i\tt$. Here the numbers $\tt$ may be (non-real) complex numbers, if the Riemann hypothesis
fails. Then  we show without any assumption the following.
\begin{maintheorem}
For $\kappa>0$  we have
\begin{equation}\label{E:main1}
\sum_{0<\Re\tt\le T}e^{2i\kappa\vartheta(\tt)}=-\frac{e^{\frac{\pi i}{4}(1-\kappa)}}{\sqrt{\kappa}}\sum_{n\le (T/2\pi)^{\kappa}}
\frac{\Lambda(n)}{n^{\frac12-\frac{1}{2\kappa}}}
e^{-2\pi i \kappa n^{1/\kappa}}+\Orden_\kappa(T^{\frac{1-\kappa}{2}}\log T)+ 
\Orden_\kappa(T^{\frac{\kappa}{2}}\log^2T).
\end{equation}
\end{maintheorem}

Applying, for $\kappa>1$, the Piatetski-Shapiro $11/12$ Theorem, we get
\[\lim_{T\to\infty}\frac{1}{N(T)}\sum_{0<\Re\tt\le T}e^{2i\kappa\vartheta(\tt)}=0,
\qquad 0<\kappa<\frac65.\]

The case $\kappa=1$ is especially interesting, since it seems to contradict the behavior 
of the
first two million zeros of $\zeta(s)$. 

Our results are also very similar to some of Schoi\ss engeier \cite{MR528875}, who
extended the cited analysis of Hardy and Littlewood. In the case $\kappa=1$ our formula gives
the following
\begin{cor} 
The Von Mangoldt function can be approximated in the following way
\[\psi(T/2\pi)=\sum_{n\le T/2\pi}\Lambda(n)=-\sum_{0<\Re\alpha\le T}e^{2 i\vartheta(\alpha)}+
\Orden(T^{\frac12}\log^2T).\]
\end{cor}
This is equivalent to  one of the results of Schoi\ss engeier, but obtained here by a simpler analysis. 

\subsection{Computational data.}
Our interest in the distribution of the $\gamma_n$ with respect to the  Gram points
originates from the observation that the angle of the curve $\Re\zeta(s)=0$  at a 
zero $\frac12+i\gamma$ (on the critical line)  with the positive real axis is equal  to
$\vartheta(\gamma)\bmod \pi$. Odlyzko's list of 
first $2,001,052$ zeros of zeta has been used to generate the following pictures 
of the distribution
of $\frac1\pi\vartheta(\gamma)\bmod\ 1$.
\begin{figure}[H]
  \includegraphics[width=6.8cm]{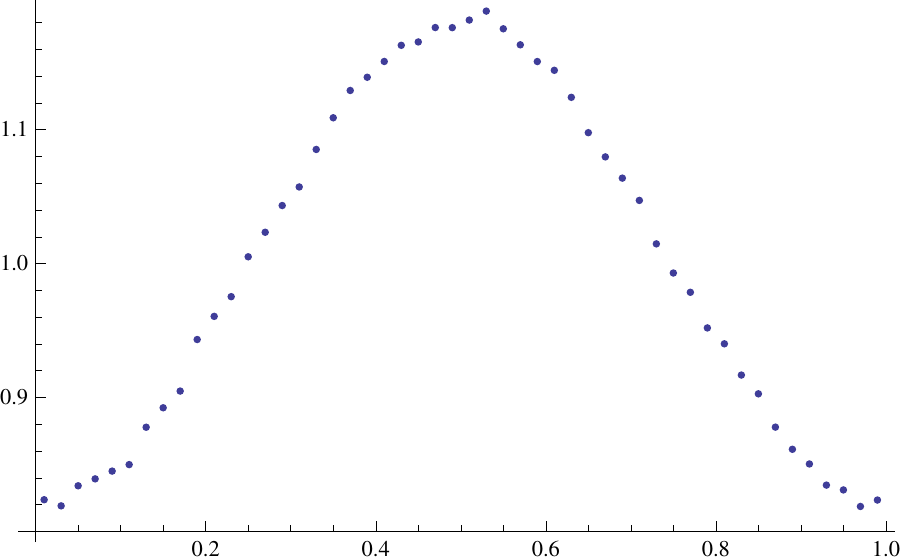}
  \includegraphics[width=6.8cm]{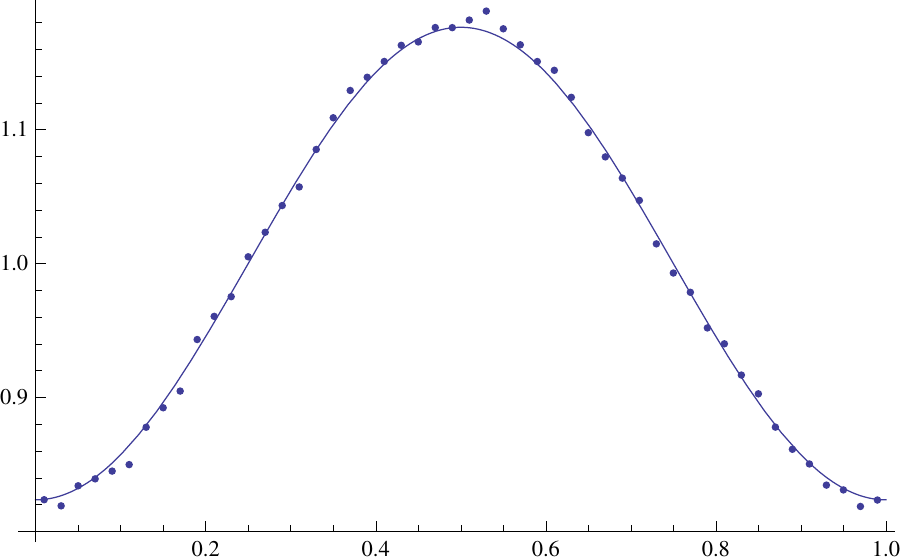}\\
 \caption{Distribution mod 1 of the normalized first $2\,001\,052$ zeros.}
  \label{F:histogram}
\end{figure}

The curve we have drawn, approximately fitting the data, is the density function
$\rho(x)=1-\frac{3}{17}\cos(2\pi x)$.
If the RH is true so that $\alpha_n=\gamma_n$ for all $n$ and this initial distribution
is maintained then we  have
\[\lim_T\frac{1}{N(T)}\sum_{0<\Re \alpha\le T}e^{2i\vartheta(\alpha)}=
\lim_N\frac{1}{N}\sum_{n=1}^N e^{2\pi i x_n}\approx\int_0^1 e^{2\pi i x}\Bigl(1-\frac3{17}
\cos2\pi x\Bigr)\,dx=-\frac{3}{34}.\]
Hence, our Corollary \ref{Cor:Limits}  shows that the trend in the behavior of 
$\gamma_n$ seen in the above figures will not be maintained for larger values of $T$. 

Odlyzko  observed that the distribution of the 
normalized zeros is nearer to a uniform distribution for higher zeros.  
But this seems to be
at odds with Titchmarsh \cite[{Theorem 10.6}]{T} who shows  that the mean values
of $Z(g_{2n})$ and  $Z(g_{2n+1})$ are equal to $2$ and $-2$, respectively.

Hence we ask  the question: \emph{It is true that the normalized zeros are 
uniformly distributed modulo $1$?}

We also comment here about another conjecture regarding the distribution of the zeros 
of zeta. Assuming the RH,
Montgomery \cite{MR0337821} proved his result about the correlation of pairs of 
zeros and stated his \emph{pair correlation conjecture}.

The differences $x_{n+1}-x_n$  of the normalized zeros satisfy 
\[x_{n+1}-x_n=\frac{1}{\pi}\int_{\gamma_n}^{\gamma_{n+1}}\vartheta'(t)\,dt.\]
Since $\vartheta'(t)=\frac12\log\frac{t}{2\pi}+\Orden(t^{-2})$ we have
\[x_{n+1}-x_n\approx \frac{\gamma_{n+1}-\gamma_n}{\pi}\frac12\log\frac{\gamma_n}{2\pi} = \delta_n.\]
So, Montgomery's conjecture is also a conjecture about our normalized zeros.  

A natural question here is:  \emph{Is the pair correlation conjecture compatible with 
a uniform distribution (mod 1) of the normalized zeros?}

\section{Notations and tools.}\label{known facts}

When possible we follow the standard notations.
As in Titchmarsh~\cite[{section 9.4}]{T} the zeros $\beta+i\gamma$ of $\zeta(s)$ 
with $\gamma>0$ are arranged in a  sequence $\rho_n=\beta_n+i\gamma_n$ so that
$\gamma_{n+1}\ge\gamma_n$. We will not assume the Riemann hypothesis (RH for short),
and following Riemann~\cite{R} define $\tt_n$ by $\rho_n=\frac12+i\tt_n$. The 
numbers $\tt_n$ are the zeros of $\Xi(t)$ with positive real part. The RH is equivalent
to the equality $\tt_n=\gamma_n$ for all natural numbers $n$. We denote by $N(T)$,
where $T>0$, the number of zeros of $\zeta(s)$ in the rectangle $0<\sigma<1$, 
$0\le t\le T$. 

The functional equation of $\zeta(s)$ can be written as 
\begin{equation}
\zeta(s)=\chi(s)\zeta(1-s),\quad \chi(s)=\pi^{s-\frac12}\frac{\Gamma(\frac12-\frac12s)}
{\Gamma(\frac12s)}=2^{s-1}\pi^s \sec\medio\pi s/\Gamma(s).
\end{equation}
For $t$ real we have $|\chi(\frac12+it)|=1$, and 
there exist two real and  real analytic functions $\vartheta(t)$ and $Z(t)$
(see Titchmarsh~\cite[{section 4.17}]{T}) such that
\begin{equation}
\zeta(\medio+it)=e^{-i\vartheta(t)}Z(t),\quad \chi(\medio+it)=e^{-2i\vartheta(t)}.
\end{equation}
The function $\chi(s)$ has poles for $s=2n+1$ with $n=0$, $1$, \dots\ and zeros for
$s=-2n$. Let $\Omega$ be the complex plane $\C$ with two cuts along the half-lines
$(-\infty,0]$ and $[1,\infty)$. The function $\chi(s)$ is analytic on the simply 
connected $\Omega$ and 
does not vanish there. So we may define $\log\chi(s)$ in $\Omega$ in such a way that
\begin{equation}
-\log\chi(\medio+it)=2i\vartheta(t),\qquad \vartheta(0)=0.
\end{equation}
The function $\vartheta(t)$ is extended in this way to all $-i(\Omega-\frac12)$ as an analytic
function. 
Also we fix the meaning of 
\begin{equation}
\chi(s)^{-\kappa}:=e^{-\kappa\log\chi(s)}=e^{2i\kappa
\vartheta(\tau)}\quad\text{ where }\quad s=\medio+i\tau\in\Omega. 
\end{equation}
\medskip

\begin{definition}
For any non-trivial zero $\rho=\beta+i\gamma=\frac12+i\tt$ we define the normalized zero 
as
\begin{equation}
x=\frac{1}{\pi}\vartheta(\tt).
\end{equation}
Also let $x_n$ be the normalized zero corresponding to $\rho_n=\beta_n+i\gamma_n$. 
\end{definition}

The function $\vartheta(t)$ is strictly increasing for $t>6.28984\dots$
For integral $k\ge-1$ the Gram point $t=g_k$($>7$) is defined as the unique solution
of $\vartheta(t)=k\pi$ (see \cite[{p. ~126}]{E}).  

In the interval $[0,T]$ there are approximately 
as many Gram points as zeros $\beta+i\gamma$ of
$\zeta(s)$ with $0<\gamma\le T$.  Gram's ``law'' (to which there are many 
exceptions) states that in each Gram interval $(g_n,g_{n+1})$ there is a zero of 
$\zeta(s)$. If Gram's law were generally true, then the RH would be true, the zeros would be simple
and $\gamma_n$ would be an element of $(g_{n-2},g_{n-1})$. Of course Gram's 
law is not true, but it is still  a good
heuristic to locate the zeros of $\zeta(s)$ for relatively small  $t$  which
can be reached by our 
computers.

\medskip
Also, it is well known that
in each interval $(T,T+1)$, with $T\ge2$ we can select a number $T'$ such that if 
$\gamma$ is the ordinate of any zero of $\zeta(s)$ then $|T'-\gamma|\gg \frac{1}{\log T}$.

\bigskip

From the book by Huxley we quote two lemmas which will be essential in our proof
\cite[{p.~88}]{MR1420620}.

\begin{lemma}[First Derivative Test]\label{1dertest}
Let $f(x)$ be real and differentiable on the open interval $(\alpha,\beta)$ with 
$f'(x)$ monotone and $f'(x)\ge\mu>0$ on $(\alpha,\beta)$. Let $g(x)$ be real, and let
$V$ be the total variation of $g(x)$ on the closed interval $[\alpha,\beta]$
plus the maximum modulus of $g(x)$ on $[\alpha,\beta]$. Then
\begin{displaymath}
\Bigl|\int_\alpha^\beta g(x) \exp(2\pi i f(x))\,dx\Bigr|\le \frac{V}{\pi\mu}.
\end{displaymath}
\end{lemma}

\begin{lemma}[Second Derivative Test]\label{2dertest}
 Let $f(x)$ be real and twice differentiable
on the open interval $(\alpha,\beta)$ with $f''(x)\ge\lambda>0$ on $(\alpha,\beta)$. 
Let $g(x)$ be real, and let $V$ be the total variation of $g(x)$ on the closed interval
$[\alpha,\beta]$ plus the maximum modulus of $g(x)$ on $[\alpha,\beta]$. Then
\begin{displaymath}
\Bigl|\int_\alpha^\beta g(x) \exp(2\pi i f(x))\,dx\Bigr|\le \frac{4V}{\sqrt{\pi\lambda}}.
\end{displaymath}
\end{lemma}

The next two lemmas can be inferred from  Levinson \cite{L} and Gonek \cite{MR728143}.

\begin{lemma}\label{L:levinson1}
Let $\kappa_0>0$  and $K\subset \R$ a compact set be given. Then there exist 
constants $c>0$,  $C>0$  
such that for any $r>1$,  $a\in K$ and $\kappa\ge \kappa_0$, we have
\begin{displaymath}
\int_{r(1-c)}^{r(1+c)}x^a \exp\Bigl\{2\pi i\Bigl(\kappa x\log\frac{x}{er}\Bigr)\Bigr\}\,dx=
\kappa^{-\frac12}r^{a+\frac12}e^{\frac{\pi i}{4}}e^{-2\pi i \kappa r}+ R,
\end{displaymath}
with $|R|\le C r^a$. 
\end{lemma}

\begin{lemma}\label{L:levinson2}
Let $\kappa_0>0$  and $K\subset \R$ a compact set be given. Then there exist constants $c>0$,  $C>0$  
such that for any $r>1$, $\kappa\ge\kappa_0$,  $a\in K$ and $\frac{r}{2}\le A<B\le 2r$ we have
\begin{displaymath}
\int_A^B x^a \exp\Bigl\{2\pi i\Bigl(\kappa x\log\frac{x}{er}\Bigr)\Bigr\}\,dx= I_0+
R_1+R_2+R_3,
\end{displaymath}
where
\begin{displaymath}
|R_1|\le C r^a,\quad |R_2|\le C\frac{r^{a+1}}{|A-r|+r^{1/2}},
\quad |R_3|\le C\frac{r^{a+1}}{|B-r|+r^{1/2}},
\end{displaymath}
and where $I_0=\kappa^{-\frac12}r^{a+\frac12}e^{\frac{\pi i}{4}}e^{-2\pi i \kappa r}$ for $A\le r\le B$ 
and $0$ in all other cases.
\end{lemma}

Now we state the best zero-free region known.  
A proof can be found in the book of Ivi\'c \cite[{Thm.~6.1}]{MR792089}.

\begin{theorem}\label{zero-free}
There is an absolute constant $C>0$ such that $\zeta(s)\ne0$ for 
\begin{equation}
\sigma\ge 1-C(\log t)^{-\frac23}(\log\log t)^{-\frac13}\qquad  (t\ge t_0).
\end{equation}
\end{theorem}

\begin{lemma}\label{boundvartheta}
Let $\rho=\beta+i\gamma$ with $\beta\in(0,1)$ and $\gamma>0$ 
and define $\tt$ by $\rho=\frac12+i\tt$. Then for any $\kappa>0$ we have
\begin{equation}
e^{2i\kappa \vartheta(\tt)}=
\Bigl(\frac{\gamma}{2\pi}\Bigr)^{\kappa(\beta-\frac12)}
\exp\Bigl\{i\Bigl(\kappa\gamma\log\frac{\gamma}{2\pi}-\kappa\gamma-\frac{\kappa\pi}{4}\Bigr)
\Bigr\}
(1+\Orden_{\kappa}(\gamma^{-1})).
\end{equation}
\end{lemma}

\begin{proof}
This follows easily from   Titchmarsh \cite[{eq.~(4.12.3)}]{T}.
\end{proof}

We will use the following Theorem of Piatetski-Shapiro \cite{MR0059302}.

\begin{theorem}\label{Piatetski}
For $\varepsilon>0$, $\frac{2}{3}<\gamma<1$ and all $k$ with $1\le k\le x^{1-\gamma}\log^2 x$,  we have
\begin{equation}
\sum_{p\le x} e^{2\pi i k p^\gamma}\ll x^{\frac{11}{12}+\varepsilon}.
\end{equation}
\end{theorem}

The exponent 11/12 in this Theorem has been improved, but with a smaller 
range of $\gamma$.
For our needs the range is important. Therefore, we will use this theorem as stated. 

\section{Main theorem}

\begin{theorem}\label{maintheorem}
For $\kappa>0$  
\begin{equation}\label{E:main}
\sum_{0<\Re\tt<T}e^{2i\kappa\vartheta(\tt)}=-
\frac{e^{\frac{\pi i}{4}(1-\kappa)}}{\sqrt{\kappa}}
\sum_{n<(T/2\pi)^{\kappa}}
\frac{\Lambda(n)}{n^{\frac12-\frac{1}{2\kappa}}}
e^{-2\pi i \kappa n^{1/\kappa}}+\Orden_\kappa(T^{\frac{1-\kappa}{2}}\log T)+ 
\Orden_\kappa(T^{\frac{\kappa}{2}}\log^2T).
\end{equation}
\end{theorem}

\begin{proof}
There exists $T'$ such that $T<T'<T+1$ with $|T'-\gamma|\gg1/\log T$ for any 
ordinate $\gamma$ of a zero of $\zeta(s)$ and such that 
\begin{displaymath}
\sum_{0<\Re\tt\le T}e^{2i\kappa\vartheta(\tt)} =\sum_{0<\Re\tt\le T'}e^{2i\kappa\vartheta(\tt)}
+\Orden_\kappa(T^{\frac{\kappa}{2}}).
\end{displaymath}
In fact, here the difference between the two sums is composed of at most of $C\log T$ terms.
Lemma \ref{boundvartheta}  yields
\begin{displaymath}
\Bigl|\sum_{T<\Re\tt\le T'}e^{2i\kappa\vartheta(\tt)}\Bigl|\le \sum_{T<\gamma\le T'}
\Bigl(\frac{\gamma}{2\pi}\Bigr)^{\kappa(\beta-\frac12)}.
\end{displaymath}
Applying  Theorem \ref{zero-free} we have
\begin{multline*}
\Bigl|\sum_{T<\Re\tt\le T'}e^{2i\kappa\vartheta(\tt)}\Bigl|\le C\log T
\Bigl(\frac{T'}{2\pi}\Bigr)^{\kappa(\frac12-c(\log T')^{-\frac23}
(\log\log T')^{-\frac13})}\\
\ll_\kappa C(\log T)T^{\frac{\kappa}{2}} 
e^{-c \kappa (\log T')^{\frac13}(\log\log T')^{-\frac13}}=\Orden_\kappa(T^{\frac{\kappa}{2}}).
\end{multline*}

Also 
\begin{multline*}
\Bigl|
\sum_{(T/2\pi)^\kappa<n\le (T'/2\pi)^\kappa}\frac{\Lambda(n)}{n^{\frac12-\frac{1}{2\kappa}}}
e^{-2\pi i \kappa n^{\frac1\kappa}}\Bigr|\ll_\kappa (\log T) T^{\frac12-\frac\kappa2}(T'^\kappa-T^\kappa)
\\
\ll_\kappa (\log T)T^{\frac12+\frac\kappa2}\Bigl|1-\Bigl(\frac{T+1}{T}\Bigr)^\kappa\Bigr|
\ll_\kappa (\log T)T^{\frac\kappa2-\frac12}.
\end{multline*}

Therefore, replacing $T$ by $T'$ if needed,  we may assume for the rest of the proof 
that $T$ satisfies $|T-\gamma|\gg 1/\log T$
for any ordinate $\gamma$ of a zero of $\zeta(s)$. 
\medskip

Since (cf.~Davenport \cite[{p.~80}]{MR1790423})
\begin{equation}
\frac{\zeta'(s)}{\zeta(s)}=B-\frac{1}{s-1}+\sum_{n=1}^\infty\Bigl(\frac{1}{s+2n}-\frac{1}{2n}\Bigr)
+\sum_{\rho}\Bigl(\frac{1}{s-\rho}+\frac{1}{\rho}\Bigr)
\end{equation}
we have by Cauchy's Theorem,
\begin{equation}
U(T):=\sum_{0<\Re\tt<T}e^{2i\kappa\vartheta(\tt)}=
\frac{1}{2\pi i}\int_{C_T}\frac{\zeta'(s)}{\zeta(s)}\chi(s)^{-\kappa}\,ds.
\end{equation}
Here the path of integration  $C_T$ is the   boundary  of the rectangle $(\sigma_0,\sigma_1)\times(2\pi,T)$  
with $1<\sigma_1<3/2$ and $T$  with $|T-\gamma_n|\gg 1/\log T$ for all $n$.  We will take
$\sigma_1=1+\frac{1}{\log T}$, but we maintain the simpler notation $\sigma_1$.
The restriction $\sigma_1<3/2$ allows one to obtain explicit bounds (independent of
$\sigma_1$) on all pertinent inequalities.

The value of $\sigma_0$ depends on $\kappa$. We will take 
$\sigma_0=\frac12-\frac2\kappa$ for $0<\kappa<\frac43$ and $\sigma_0=-1$ when 
$\kappa\ge \frac43$. In this way $\sigma_0\le -1$ in all cases.

Then we have
\begin{multline*}
U(T)=\frac{1}{2\pi i}\int_{\sigma_0+2\pi i}^{\sigma_1+2\pi i}\frac{1}{\chi(s)^{\kappa}}
\frac{\zeta'(s)}{\zeta(s)}\,ds+\frac{1}{2\pi i}\int_{\sigma_1+2\pi i}^{\sigma_1+Ti}\cdots-
\frac{1}{2\pi i}\int_{\sigma_0+Ti}^{\sigma_1+Ti}\cdots-\frac{1}{2\pi i}\int_{\sigma_0+2\pi i}^{\sigma_0+Ti}\cdots\\
:=U_1(T)+U_2(T)-U_3(T)-U_4(T).
\end{multline*}
Lemmas \ref{B1}, \ref{B3} and \ref{B4}
yield
\[U(T)=\Orden_\kappa(1)+U_2(T)+\Orden_\kappa(T^{\frac{\kappa}{2}}\log T)+
\Orden_\kappa(1).\]
Now we apply Lemma \ref{B2} and we see that $U_2(T)$ is 
equal to the sum on the right in \eqref{E:main} plus the   remainders $\Orden_\kappa(T^{\frac{\kappa}{2}}\log^2 T)$
and  $\Orden_\kappa(T^{\frac{1-\kappa}{2}}\log T)$. 
Therefore, we have our result.
\end{proof}

\begin{corollary}\label{Cor:Limits}
For $0<\kappa<\frac{6}{5}$ we have 
\begin{equation}
\lim_{T\to \infty}\frac{1}{N(T)}\sum_{0<\Re\tt\le T}e^{2i\kappa\vartheta(\tt)}=0.
\end{equation}
For $\kappa=0$ the above limit is  easily seen to be $1$.
\end{corollary}
\begin{proof}
We have
$N(T)=\card\{\tt: 0<\Re\tt\le T\}=\Orden(T\log T)$. For $0<\kappa<1$, the trivial bound yields
\[\sum_{n\le(T/2\pi)^{\kappa}}
\frac{\Lambda(n)}{n^{\frac12-\frac{1}{2\kappa}}}
e^{-2\pi i \kappa n^{1/\kappa}}=\Orden_\kappa(T^{\frac{\kappa}{2}+\frac12}\log T)\]
and the limit is easily shown to be $0$. 

In the case $\kappa=1$ we apply  Theorem \ref{maintheorem},
and observe that in this case
\[\Bigl|\frac{e^{\pi i/4}}{\sqrt{\kappa}}\sum_{n\le(T/2\pi)^{\kappa}}
\frac{\Lambda(n)}{n^{\frac12-\frac{1}{2\kappa}}}
e^{-2\pi i \kappa n^{1/\kappa}}\Bigr|\le \sum_{n\le T/2\pi}\Lambda(n)=\Orden(T).\]
Therefore, 
\[\frac{1}{N(T)}\sum_{0<\Re\tt\le T}e^{2i\kappa\vartheta(\tt)}=\Orden(1/\log T).\]

For $1<\kappa<\frac{3}{2}$, the  Piatetski-Shapiro Theorem \ref{Piatetski} 
with $k=\kappa$, 
and $\gamma=1/\kappa$   yields, for any $\varepsilon>0$,
\[\sum_{n\le x}\Lambda(n)e^{2\pi i \kappa n^\gamma}=\Orden(x^{\frac{11}{12}+\varepsilon}).\]
Partial summation then yields  
\[\sum_{n\le (T/2\pi)^{\kappa}}
\frac{\Lambda(n)}{n^{\frac12-\frac{1}{2\kappa}}}
e^{-2\pi i \kappa n^{1/\kappa}}=\Orden(T^{\frac{5\kappa}{12}+\frac{1}{2}+\varepsilon \kappa}).\]
It follows that 
\[\frac{1}{N(T)}\sum_{0<\Re\tt\le T}e^{2i\kappa\vartheta(\tt)}=
\Orden(T^{\frac{5\kappa}{12}-\frac{1}{2}+\varepsilon\kappa}/\log T)+\Orden(T^{\frac{\kappa}{2}-1}\log T).\]
So  the limit is $0$ for $1<\kappa<\frac{6}{5}<\frac{3}{2}$.
\end{proof}

\section{Bounds.}

\subsection{Bound of the bottom integral.}

\begin{lemma}\label{B1}
Uniformly for  all $\sigma_1\in(1,3/2)$ 
\begin{equation}
U_1(T)=\frac{1}{2\pi i}\int_{\sigma_0+2\pi i}^{\sigma_1+2\pi i}\frac{1}{\chi(s)^{\kappa}}
\frac{\zeta'(s)}{\zeta(s)}\,ds=\Orden_\kappa(1).
\end{equation}
\end{lemma}

\begin{proof}
$U_1(T)$ is a well defined and continuous function of $\sigma_1\in[1,3/2]$ (it does not depend on $T$).
\end{proof}

\subsection{Bound of the top integral.}

\begin{lemma}\label{B3}
Let $T$ be such that $|T-\gamma_n|\gg 1/\log T$, and let $\sigma_1=1+\frac{1}{\log T}$. Then 
\begin{equation}
U_3(T)=\frac{1}{2\pi i}\int_{\sigma_0+ i T}^{\sigma_1+i T}\frac{1}{\chi(s)^{\kappa}}
\frac{\zeta'(s)}{\zeta(s)}\,ds=\Orden_\kappa(T^{\kappa/2}\log T).
\end{equation}
\end{lemma}

\begin{proof}
We apply Titchmarsh \cite[{eq.~(4.12.3)}]{T} so that 
\[\chi(s)=\Bigl(\frac{2\pi}{t}\Bigr)^{\sigma+it-\frac12} e^{i(t+\frac14\pi)}
\bigl\{1+\Orden(t^{-1})\bigr\}\]
on any strip $\alpha\le\sigma\le\beta$ and for $t\to+\infty$.
Therefore, we will have
\begin{equation}\label{chikappa}
\chi(s)^{-\kappa}=\Bigl(\frac{t}{2\pi}\Bigr)^{\kappa(\sigma+it-\frac12)} e^{-i
\kappa(t+\frac14\pi)}
\bigl\{1+\Orden_\kappa(t^{-1})\bigr\}.
\end{equation}
It follows that for $s=\sigma+iT$ with $\sigma_0<\sigma<\sigma_1$ we will have
\[|\chi(\sigma+iT)|^{-\kappa}\le C T^{\kappa(\sigma-\frac12)}.\]
We choose $T$ satisfying  $|T-\gamma|\gg 1/\log T$,  
so that by applying Theorem 9.6(A) of Titchmarsh
we  get on the segment $s=\sigma+iT$ with $-1<\sigma<\sigma_1$ that 
\[\frac{\zeta'(s)}{\zeta(s)}=\Orden(\log^2T).\]
By 
Ingham \cite[{Theorem 27 p.~73}]{MR1074573} this extends to the entire segment $\sigma_0<\sigma<\sigma_1$.  Then, with the constant $C$
depending on $\kappa$.
\[|U_3(T)|\le C\int_{\sigma_0}^{\sigma_1} T^{\kappa(\sigma-1/2)}\log^2(T)\,d\sigma\le 
C\frac{T^{\kappa(\sigma_1-1/2)}}{\kappa\log T}\log^2T.\]
Taking $\sigma_1=1+\frac{1}{\log T}$ we get
\[|U_3(T)|=\Orden_{\kappa} (T^{\kappa/2}\log T).\]
\end{proof}

\subsection{Bound of the left integral.}

\begin{lemma}\label{B4}
For $0\le\kappa$ we have
\begin{equation}
U_4(T)=\frac{1}{2\pi i}\int_{\sigma_0+ 2\pi i}^{\sigma_0+i T}\frac{1}{\chi(s)^{\kappa}}
\frac{\zeta'(s)}{\zeta(s)}\,ds=\Orden_\kappa(1).
\end{equation}
\end{lemma}

\begin{proof}
We integrate along the line $s=\sigma_0+it$ with $2\pi<t<T$. 
So we may apply Ingham \cite[{Theorem 27 p.~73}]{MR1074573}
so that 
\[\frac{\zeta'(s)}{\zeta(s)}=\Orden(\log t).\]

Also, applying \eqref{chikappa} we get
\begin{equation}
U_4(T)\ll_{\kappa}\int_{2\pi}^T t^{\kappa(\sigma_0-\frac12)}\log t\,dt.
\end{equation}
The choice of $\sigma_0$ ($\sigma_0=\frac12-\frac2\kappa$ for $0<\kappa<\frac43$, 
and $\sigma_0=-1$ when 
$\kappa\ge \frac43$) guarantees that $\kappa(\sigma_0-\frac12)\le -2$.
Therefore, the integral is bounded.
\end{proof}

\subsection{Bound of the right integral.}

\begin{lemma}\label{B2}
Taking $\sigma_1=1+\frac{1}{\log T}$ we have
\begin{multline}
U_2(T)=\frac{1}{2\pi i}\int_{\sigma_1+ 2\pi i}^{\sigma_1+i T}\frac{1}{\chi(s)^{\kappa}}
\frac{\zeta'(s)}{\zeta(s)}\,ds\\
=-\frac{e^{\frac{\pi i}{4}(1-\kappa)}}{\sqrt{\kappa}}\sum_{n<(T/2\pi)^{\kappa}}
\frac{\Lambda(n)}{n^{\frac12-\frac{1}{2\kappa}}}
e^{-2\pi i \kappa n^{1/\kappa}}+\Orden_\kappa(T^{\frac{1-\kappa}{2}}\log T)+ 
\Orden_\kappa(T^{\frac{\kappa}{2}}\log^2T).
\end{multline}
\end{lemma}

\begin{proof}
We have taken $\sigma_1>1$ in order to apply the expression as a Dirichlet series. So
\[U_2(T)=-\sum_{n=1}^\infty \frac{\Lambda(n)}{n^{\sigma_1}}\frac{1}{2\pi i}
\int_{\sigma_1+ 2\pi i}^{\sigma_1+i T}\frac{1}{\chi(s)^{\kappa}}e^{-i t\log n}
\,ds\]
with $s=\sigma_1+it$.  Therefore, by \eqref{chikappa}
\[U_2(T)=-\frac{1}{2\pi}\sum_{n=1}^\infty \frac{\Lambda(n)}{n^{\sigma_1}}\int_{2\pi}^T e^{-it\log n}
\Bigl(\frac{t}{2\pi}\Bigr)^{\kappa(\sigma_1+it-\frac12)} e^{-i
\kappa(t+\frac14\pi)}
V(t)\,dt\]
where $V(t)=1+\Orden_\kappa(t^{-1})$.  Then 
\[U_2(T)=-\frac{1}{2\pi}\sum_{n=1}^\infty \frac{\Lambda(n)}{n^{\sigma_1}}\int_{2\pi}^T e^{-it\log n}
\Bigl(\frac{t}{2\pi}\Bigr)^{\kappa(\sigma_1+it-\frac12)} e^{-i
\kappa(t+\frac14\pi)}
\,dt + R\]
where $R$ is the error term. Then
\[|R|\ll_\kappa\sum_{n=2}^\infty \frac{\Lambda(n)}{n^{\sigma_1}}\int_{2\pi}^T
t^{\kappa(\sigma_1-\frac12)-1}\,dt\ll_\kappa
\sum_{n=2}^\infty \frac{\Lambda(n)}{n^{\sigma_1}} T^{\kappa(\sigma_1-\frac12)}=
-T^{\kappa(\sigma_1-\frac12)}\frac{\zeta'(\sigma_1)}{\zeta(\sigma_1)}.\]
Therefore, taking $\sigma_1=1+\frac{1}{\log T}$ we get 
$R=\Orden_\kappa(T^{\kappa/2}\log T)$. 

It remains to compute
\begin{multline*}
V_2(T):=\frac{1}{2\pi}\sum_{n=1}^\infty \frac{\Lambda(n)}{n^{\sigma_1}}\int_{2\pi}^T e^{-it\log n}
\Bigl(\frac{t}{2\pi}\Bigr)^{\kappa(\sigma_1+it-\frac12)} e^{-i
\kappa(t+\frac14\pi)}
\,dt=\\
= e^{-i\kappa\pi/4}\sum_{n=1}^\infty \frac{\Lambda(n)}{n^{\sigma_1}}
\int_{1}^{T/2\pi} x^{\kappa(\sigma_1-\frac12)} 
\exp\Bigl\{2\pi i\Bigl(\kappa x \log x-\kappa x-x\log n \Bigr)\Bigr\}\,dx.
\end{multline*}
These integrals are classical stationary phase integrals. We will apply the 
theorems in Huxley
\cite{MR1295511}, or \cite{MR1420620}, and Lemma \ref{L:levinson2}.  
The stationary phase occurs when $\kappa\log x-\log n=0$. That is for $x=n^{1/\kappa}$.
We subdivide  each integral in  three parts, by dividing the interval of integration
$I=(1,T/2\pi)$ in three parts 
\[I_0:=I\cap(\medio n^{1/\kappa},2n^{1/\kappa}),\qquad I_1:=I\smallsetminus (\medio n^{1/\kappa},+\infty),\qquad I_2=I\smallsetminus(-\infty, 2n^{1/\kappa}).\]
It is easy to see that $I_0\cup I_1\cup I_2=I$ is a partition.  Some of these three
intervals may be empty. The partition depends on $n$, and  is different for each 
integral, but we prefer to maintain a simple notation.

In this way $V_2(T)=V_{2,0}(T)+V_{2,1}(T)+V_{2,2}(T)$ is subdivided in three parts
\[V_{2,j}(T):= e^{-i\kappa\pi/4}\sum_{n=1}^\infty \frac{\Lambda(n)}{n^{\sigma_1}}
\int_{I_j} x^{\kappa(\sigma_1-\frac12)} 
\exp\Bigl\{2\pi i\Bigl(\kappa x \log x-\kappa x-x\log n \Bigr)\Bigr\}\,dx.\]

\subsubsection{Bound of $V_{2,1}(T)$}  In this case the interval of integration
$I_1:=I\smallsetminus (\frac12 n^{1/\kappa},+\infty)$ does not contain the stationary
point $n^{1/\kappa}$, and  for any point $x$ in this interval
we have $x<\frac12n^{\frac{1}{\kappa}}$ so that 
\[f'(x)=\kappa\log x-\log n<-\kappa\log 2<0.\]
We may apply Lemma \ref{1dertest}. 
Since $I_1\subset(1,T/2\pi)$, the constant $V$ in the  lemma is less than 
\[V\le 2\cdot \Bigl(\frac{T}{2\pi}\Bigr)^{\kappa(\sigma_1-\frac12)}.\]
So, we get
\[|V_{2,1}(T)|\le\sum_{n=1}^\infty \frac{\Lambda(n)}{n^{\sigma_1}}
\frac{2\left(\frac{T}{2\pi}\right)^{\kappa(\sigma_1-\frac12)}}{\pi\kappa\log 2}
\ll_\kappa \Bigl(\frac{T}{2\pi}\Bigr)^{\kappa(\sigma_1-\frac12)}
\frac{|\zeta'(\sigma_1)|}{\zeta(\sigma_1)}.\]
We now choose $\sigma_1=1+\frac{1}{\log T}$ and get 
\[|V_{2,1}(T)|\ll_\kappa T^{\frac{\kappa}{2}}\log T.\]

\subsubsection{Bound of $V_{2,2}(T)$} In this case the interval of integration is
$I_2=I\smallsetminus(-\infty, 2n^{1/\kappa})$. Hence, the stationary point 
$n^{\frac{1}{\kappa}}\notin I_2$ and we may apply Lemma \ref{1dertest}  again.
For $x\in I_2$ we have $x\ge 2n^{\frac{1}{\kappa}}$ so that
\[f'(x)=\kappa\log x-\log n \ge \kappa\log 2>0.\]
If $I_2$ is non empty we have $2n^{\frac{1}{\kappa}}\le \frac{T}{2\pi}$. 
The $V$ of Lemma \ref{1dertest} is the same as in the previous case so that
\[|V_{2,2}(T)|\le \sum_{n\le (T/4\pi)^\kappa}\frac{\Lambda(n)}{n^{\sigma_1}}
\frac{2\left(\frac{T}{2\pi}\right)^{\kappa(\sigma_1-\frac12)}}{\pi\kappa\log 2}\]
and, as before, we get the same bound
\[|V_{2,2}(T)|\ll_\kappa T^{\frac{\kappa}{2}}\log T.\]

\subsubsection{Bound of $V_{2,0}(T)$}  

We will apply Lemma \ref{L:levinson2}.
In our case $a=\kappa(\sigma_1-\frac12)$ and 
$f(x)=\kappa x\log x-\kappa x-x\log n=\kappa x\log\frac{x}{e n^{1/\kappa}}$
so that $r=n^{1/\kappa}$.
Since the interval
of integration is given by $I_0=(1,T/2\pi)\cap
(\medio n^{1/\kappa},2n^{1/\kappa})$, we will have $r\in I_0$, only in the case 
$1<n^{\frac{1}{\kappa}}<\frac{T}{2\pi}$. 

For $n>(T/\pi)^\kappa$ the interval $I_0=\emptyset$. 
Therefore,
\[V_{2,0}(T)=
e^{-i\kappa\pi/4}\sum_{n\le (T/\pi)^\kappa} \frac{\Lambda(n)}{n^{\sigma_1}}
\int_{I_0} g_n(x)\exp(2\pi i f_n(x))\,dx\]
so that  Lemma \ref{L:levinson2} yields
\begin{equation}\label{E:inter}
V_{2,0}(T)=e^{-i\kappa \pi/4}\sum_{n<(T/2\pi)^\kappa} \frac{\Lambda(n)}{n^{\sigma_1}}\kappa^{-\frac12}n^{\sigma_1-\frac12+\frac{1}{2\kappa}}
e^{\frac{\pi i}{4}}e^{-2\pi i \kappa n^{1/\kappa}}+R
\end{equation}
where $R$ is bounded by 
\[R\ll \sum_{n< (T/\pi)^\kappa}\frac{\Lambda(n)}{n^{\sigma_1}}\Bigl( 
n^{\sigma_1-\frac12}+\frac{n^{\sigma_1-\frac12+\frac{1}{\kappa}}}
{|A_n-n^{\frac{1}{\kappa}}|+n^{\frac{1}{2\kappa}}}+
\frac{n^{\sigma_1-\frac12+\frac{1}{\kappa}}}
{|B_n-n^{\frac{1}{\kappa}}|+n^{\frac{1}{2\kappa}}}\Bigr)\]
where $(A_n,B_n)=(1,T/2\pi)\cap
(\medio n^{1/\kappa},2n^{1/\kappa})$. 

By partial summation we get
\[\sum_{n< (T/\pi)^\kappa}\frac{\Lambda(n)}{n^{\frac12}}=\Orden_\kappa(T^{\frac{\kappa}{2}}).\]
It is easy to see that $A_n=1$ for $n\le2^{\kappa}$ and $A_n=\frac12n^{\frac{1}{\kappa}}$
for $n\ge2^\kappa$. Then 
\[\sum_{n< (T/\pi)^\kappa}\frac{\Lambda(n)}{n^{\frac12}}\frac{n^{\frac{1}{\kappa}}}
{|A_n-n^{\frac{1}{\kappa}}|+n^{\frac{1}{2\kappa}}}\ll_\kappa
\sum_{n< (T/\pi)^\kappa}\frac{\Lambda(n)}{n^{\frac12}}=\Orden_\kappa(T^{\frac{\kappa}{2}}).\]

We have $B_n=2n^{\frac{1}{\kappa}}$ if $n\le (T/4\pi)^\kappa$ and 
$B_n=T/2\pi$ if $n\ge (T/4\pi)^\kappa$. Therefore,
\begin{multline*}
\sum_{n< (T/\pi)^\kappa}\frac{\Lambda(n)}{n^{\frac12}}\frac{n^{\frac{1}{\kappa}}}
{|B_n-n^{\frac{1}{\kappa}}|+n^{\frac{1}{2\kappa}}}\ll_\kappa\\
\ll_\kappa
\sum_{n\le  (T/4\pi)^{\kappa}}\frac{\Lambda(n)}{n^{\frac12}}+
\sum_{(T/4\pi)^\kappa<n< (T/\pi)^\kappa}\frac{\Lambda(n)}{n^{\frac12}}
\frac{n^{\frac{1}{\kappa}}}
{|T/2\pi-n^{\frac{1}{\kappa}}|+n^{\frac{1}{2\kappa}}}\ll_\kappa\\
\ll_\kappa
T^{\frac{\kappa}{2}}+ T^{\frac{1-\kappa}{2}}\log T+ T^{\frac{\kappa}{2}}\log^2T.
\end{multline*}
(See Lemma \ref{lastbound} for the last step.)

Hence, \eqref{E:inter} implies
\[V_{2,0}(T)=\frac{e^{\frac{\pi i}{4}(1-\kappa)}}{\sqrt{\kappa}}\sum_{n<(T/2\pi)^{\kappa}}
\frac{\Lambda(n)}{n^{\frac12-\frac{1}{2\kappa}}}
e^{-2\pi i \kappa n^{1/\kappa}}+\Orden_\kappa(T^{\frac{1-\kappa}{2}}\log T)+ 
\Orden_\kappa(T^{\frac{\kappa}{2}}\log^2T).\]

\subsubsection{End of the proof of Lemma \ref{B2}}
We saw that $U_2(T)=V_2(T)+\Orden_\kappa(T^{\frac{\kappa}{2}}\log T)$,
so that
\begin{multline*}
V_2(T)=V_{2,0}(T)+V_{2,1}(T)+V_{2,2}(T)=\frac{e^{\frac{\pi i}{4}(1-\kappa)}}{\sqrt{\kappa}}\sum_{n<(T/2\pi)^{\kappa}}
\frac{\Lambda(n)}{n^{\frac12-\frac{1}{2\kappa}}}
e^{-2\pi i \kappa n^{1/\kappa}}+\\
+
\Orden_\kappa(T^{\frac{\kappa}{2}}\log T)+\Orden_\kappa(T^{\frac{1-\kappa}{2}}\log T)+
\Orden_\kappa(T^{\frac{\kappa}{2}}\log^2T).
\end{multline*}
\end{proof}

\begin{lemma}\label{lastbound}
We have 
\begin{equation}
S_\kappa(T):=\sum_{(T/4\pi)^\kappa<n\le (T/\pi)^\kappa}\frac{\Lambda(n)}{n^{\frac12}}
\frac{n^{\frac{1}{\kappa}}}
{|T/2\pi-n^{\frac{1}{\kappa}}|+n^{\frac{1}{2\kappa}}}=\Orden_\kappa(T^{\frac{1-\kappa}{2}}\log T)
+\Orden_\kappa(T^{\frac{\kappa}{2}}\log^2T).
\end{equation}
\end{lemma}

\begin{proof}
We have
\begin{multline*}
S_\kappa(T)\le \frac{\log\left(\frac{T}{\pi}\right)^\kappa}
{\left(\frac{T}{4\pi}\right)^{\frac{\kappa}{2}}}
\sum_{(T/4\pi)^\kappa<n\le (T/\pi)^\kappa}
\frac{n^{\frac{1}{\kappa}}}
{|T/2\pi-n^{\frac{1}{\kappa}}|+n^{\frac{1}{2\kappa}}}\\
\le \frac{\log\left(\frac{T}{\pi}\right)^\kappa}
{\left(\frac{T}{4\pi}\right)^{\frac{\kappa}{2}}}
\Bigl(\frac{T}{\pi}\Bigr)
\sum_{(T/4\pi)^\kappa<n\le (T/\pi)^\kappa}
\frac{1}
{|T/2\pi-n^{\frac{1}{\kappa}}|+n^{\frac{1}{2\kappa}}}\\
\ll_\kappa
T^{1-\frac{\kappa}{2}}\log T
\sum_{(T/4\pi)^\kappa<n\le (T/\pi)^\kappa}
\frac{1}{|T/2\pi-n^{\frac{1}{\kappa}}|+\sqrt{T/4\pi}}.
\end{multline*}
Since the function $(|A-x^{\frac{1}{\kappa}}|+B)^{-1}$ is increasing and then decreasing, 
a geometrical argument yields that the sum  is bounded by two times the 
maximum plus an integral. Thus
\[\sum_{(T/4\pi)^\kappa<n\le (T/\pi)^\kappa}
\frac{1}{|T/2\pi-n^{\frac{1}{\kappa}}|+\sqrt{T/4\pi}}\le 2\Bigl(\frac{4\pi}{T}\Bigr)^{\frac12}+
\int_{(T/4\pi)^\kappa}^{(T/\pi)^\kappa}\frac{dx}{|T/2\pi-x^{\frac{1}{\kappa}}|+\sqrt{T/4\pi}}.\]
We change variables $x=y^\kappa$ and this yields
\begin{multline*}
\sum_{(T/4\pi)^\kappa<n\le (T/\pi)^\kappa}
\frac{1}{|T/2\pi-n^{\frac{1}{\kappa}}|+\sqrt{T/4\pi}}
\le 2\Bigl(\frac{4\pi}{T}\Bigr)^{\frac12}+
\int_{T/4\pi}^{T/\pi}\frac{\kappa y^{\kappa-1}\,dy}{|T/2\pi-y|+\sqrt{T/4\pi}}\\
\le
2\Bigl(\frac{4\pi}{T}\Bigr)^{\frac12}+
\kappa (T/\pi)^\kappa (4\pi/T)
\int_{T/4\pi}^{T/\pi}\frac{dy}{|T/2\pi-y|+\sqrt{T/4\pi}}.
\end{multline*}

It can easily be shown  that the last integral is of order $\log T$. In fact it 
is less than $\log T$
for $T\ge3$. It follows that 
\[0\le S_\kappa(T)\ll_\kappa T^{1-\frac{\kappa}{2}}\log T( T^{-\frac12}+T^{\kappa-1}\log T)\]
completing the proof of Lemma \ref{lastbound}.
\end{proof}

\section{Acknowledgement}

I wish to express my thanks to 
Jan van de Lune  ( Hallum, The Netherlands ) for his encouragements and linguistic assistance.


\end{document}